\documentclass[12pt,oneside,english]{amsart}
\usepackage[T1]{fontenc}
\usepackage[latin9]{inputenc}
\usepackage{geometry}
\usepackage{amstext}
\usepackage{amsthm}
\usepackage{amssymb}

\makeatletter
\numberwithin{equation}{section}
\numberwithin{figure}{section}
\theoremstyle{plain}
\newtheorem{thm}{\protect\theoremname}
\theoremstyle{plain}
\newtheorem{lem}[thm]{\protect\lemmaname}
\theoremstyle{plain}
\newtheorem{prop}[thm]{\protect\propositionname}
\theoremstyle{definition}
\newtheorem{defn}[thm]{\protect\definitionname}
\theoremstyle{definition}
\newtheorem{example}[thm]{\protect\examplename}
\theoremstyle{plain}
\newtheorem{cor}[thm]{\protect\corollaryname}
\theoremstyle{remark}
\newtheorem{rem}[thm]{\protect\remarkname}

\usepackage{tikz-cd}

\makeatother

\usepackage{babel}
\providecommand{\corollaryname}{Corollary}
\providecommand{\definitionname}{Definition}
\providecommand{\examplename}{Example}
\providecommand{\lemmaname}{Lemma}
\providecommand{\propositionname}{Proposition}
\providecommand{\remarkname}{Remark}
\providecommand{\theoremname}{Theorem}

\begin{document}
\title{Grothendieck Rings of Queer Lie Superalgebras}
\author{Shifra Reif}
\begin{abstract}
We determine the Grothendieck rings of the category of finite-dimensional
modules over Queer Lie superalgebras via their rings of characters.
In particular, we show that the $\mathbb{Q}$-span of the ring of
characters of the Queer Lie supergroup $Q(n)$ is isomorphic to the
ring of Laurent polynomials in $x_{1},\ldots,x_{n}$ for which the
evaluation $x_{n-1}=-x_{n}=t$ is independent of $t$. We thus complete
the description of Grothendieck rings for all classical Lie superalgebras. 
\end{abstract}

\maketitle

\section{Introduction}

The Grothendieck group is an invariant of an abelian category defined
to be the $\mathbb{Z}$-span of all objects, modulo the relation $[A]+[C]=[B]$
for every exact sequence $0\rightarrow A\rightarrow B\rightarrow C\rightarrow0$.
When the category is a tensor category, the Grothendieck group inherits
a ring structure via $[A]\cdot[B]=[A\otimes B]$. For finite-dimensional
modules over semisimple Lie algebras, the Grothendieck ring is isomorphic
to the ring of characters due to the highest weight description of
simple modules, c.f.\ e.g.\ \cite[Prop. 4.2]{SV}. For Lie superalgebras,
the ring of characters is the quotient of the Grothendieck ring by
the relation identifying a module with its parity shift. 

In this paper, we determine the ring of characters for queer Lie superalgebras.
These algebras are one of the strange series in Kac's classification
of simple Lie superalgebras \cite{K}, and is the remaining classical
series for which the Grothendieck ring is described. For the queer
Lie supergroup $Q(n)$, the description yields the following theorem. 
\begin{thm}
\label{main theorem} The $\mathbb{Q}$-span of the ring of characters
of $Q(n)$ is isomorphic to the ring 
\[
\left\{ f\in\mathbb{Q}\left[x_{1}^{\pm1},\ldots,x_{n}^{\pm1}\right]^{S_{n}}\,:\,f\mid_{x_{1}=-x_{2}=t}\text{ is independent of }t\right\} .
\]
\end{thm}

The description of the ring of characters over $\mathbb{Z}$, given
in Proposition \ref{prop:ring-over-Z}, is more subtle since the dimensions
of weight spaces are certain powers of $2$. Theorem \ref{main theorem}
is a natural extension of the description of Grothendieck rings for
the other classical Lie algebras and superalgebras: The ring of $S_{n}$-invariant
Laurent polynomials is the ring of characters of the Lie algebra $\mathfrak{gl}(n)$.
Sergeev and Veselov showed in \cite{SV} that such an evaluation condition
characterizes the ring of characters for basic-classical Lie superalgebras.
In \cite{IRS}, Im, Serganova and the author determine the Grothendieck
ring for periplectic Lie superalgebras. 

For basic-classical Lie superalgebras, the ring of characters is isomorphic
to the ring of supercharacters. For the queer Lie superalgebra this
is not the case. In fact, Cheng \cite{Che} showed that supercharacters
of a finite-dimensional nontrivial simple $\mathfrak{q}(n)$-modules
are zero. Another obstacle which arsis is the fact that the dimensions
of the highest weight spaces of a simple modules can be larger than
one. 

 To prove that all characters satisfy invariance conditions, we restrict
to low-rank subalgebras. To show the converse, namely that every invariant
function is a linear combination of characters, we use a basis of
the ring of characters which consists of Euler characteristics of
certain cohomologies introduced by Penkov and Serganova \cite{PS}.
We first prove the result for the queer Lie supergroup $Q(n)$, and
then extend it to the other queer Lie supergroups and carry it over
to Lie superalgebras. We also give a description of $J_{n}$ in terms
of invariance under the Weyl groupoid. 

\subsection*{Acknowledgments}

The author is thankful to Maria Gorelik and Vera Serganova for valuable
conversations. This project is partially supported by Israel Science
Foundations Grants No.\ 1221/17 and\ 1957/21. 

\section{The queer Lie superalgebra and its finite-dimensional representations}

We recall the structure and representation theory of the queer Lie
superalgebras. One of their main difficulty is that they do not posses
an even invariant bilinear form, moreover their Cartan subalgebras
are not purely even.

\subsection{Preliminaries}

The family of queer Lie superalgebras consist of: the subalgebra $\mathfrak{q}(n)$
of $\mathfrak{gl}(n|n)$ consisting of matrices of the form $\left[\begin{array}{cc}
A & B\\
B & A
\end{array}\right]$ where $A$, $B$ are $n\times n$ matrices,
\[
\mathfrak{sq}(n):=\left\{ \left[\begin{array}{cc}
A & B\\
B & A
\end{array}\right]\in\mathfrak{q}(n)\ :\ trB=0\right\} ,
\]
$\mathfrak{pq}(n):=\mathfrak{q}(n)/\mathbb{C}I$ and $\mathfrak{psq}(n):=\mathfrak{sq}(n)/\mathbb{C}I$.
The even parts of $\mathfrak{q}(n)$ and $\mathfrak{sq}(n)$ are isomorphic
to $\mathfrak{gl}(n)$, whereas the even parts of $\mathfrak{pq}(n)$
and $\mathfrak{psq}(n)$ are isomorphic to $\mathfrak{sl}(n)$.

The Cartan subalgebra $\mathfrak{h}=\mathfrak{h}_{\bar{0}}\oplus\mathfrak{h}_{\bar{1}}$
of $\mathfrak{q(n)}$ and $\mathfrak{sq}(n)$ consists of matrices
of the form $\left[\begin{array}{cc}
A & B\\
B & A
\end{array}\right]$ where $A$ and $B$ are diagonal $n\times n$ matrices. For $i=1,\ldots,n$,
let $H_{i}=E_{ii}+E_{n+i,n+i}$. Let $\varepsilon_{1},\ldots,\varepsilon_{n}$
denote the basis of $\mathfrak{h}_{\bar{0}}^{*}$ dual to $H_{1},\ldots,H_{n}$.
For $\mathfrak{pq}(n)$ and $\mathfrak{psq}(n)$, $\mathfrak{h}_{\bar{0}}^{*}$
and $\varepsilon_{1},\ldots,\varepsilon_{n}$ denote their images
under the natural projection. 

The set of roots is $\Phi=\left\{ \varepsilon_{i}-\varepsilon_{j}\ :\ i\ne j\right\} $
where for each root there is an odd root vector and an even root vector.
We choose the standard set of positive roots:
\[
\Phi_{\bar{0}}^{+}=\Phi_{\bar{1}}^{+}=\left\{ \varepsilon_{i}-\varepsilon_{j}\ :\ i<j\right\} .
\]
The Weyl group $W$ is isomorphic to $S_{n}$. We denote the denominators
by 
\[
R_{\bar{0}}:=\prod_{\alpha\in\Phi_{\bar{0}}^{+}}\left(1-e^{-\alpha}\right),\quad R_{\bar{1}}:=\prod_{\alpha\in\Phi_{\bar{1}}^{+}}\left(1+e^{-\alpha}\right),\quad R=\frac{R_{\bar{0}}}{R_{\bar{1}}}.
\]
Let $\rho_{0}:=(n-1)\varepsilon_{1}+(n-2)\varepsilon_{2}+\ldots+\varepsilon_{n-1}$.
Note that $e^{\rho_{0}}R_{\bar{0}}$ is $W$-anti-invariant and $e^{\rho_{0}}R_{\bar{1}}$
is $W$-invariant. 

\subsection{Finite-dimensional modules and their characters.}

We briefly recall the construction of highest weight $\mathfrak{q}(n)$-modules,
see \cite[Sec. 1.5.4]{CW} for more details. Given $\lambda\in\mathfrak{h}_{\bar{0}}^{*},$
one defines a symmetric bilinear form on $\mathfrak{h}_{\bar{1}}$
by $\left\langle v,w\right\rangle _{\lambda}:=\lambda\left(\left[v,w\right]\right)$.
Let $\mathfrak{h}_{\bar{1}}'\subseteq\mathfrak{h}_{\bar{1}}$ be a
maximal isotropic subspace with respect to $\left\langle \cdot,\cdot\right\rangle _{\lambda}$
and define $\mathfrak{h'}=\mathfrak{h}_{\bar{0}}\oplus\mathfrak{h}_{\bar{1}}'$.
The one dimensional $\mathfrak{h}_{\bar{0}}$-module $\mathbb{C}v_{\lambda}$,
defined by $xv_{\lambda}=\lambda(x)v_{\lambda}$, extends to an $\mathfrak{h}'$-module
by letting $\mathfrak{h}_{\bar{1}}'.v_{\lambda}=0$. Define the induced
$\mathfrak{h}$-module 
\[
W_{\lambda}:=\text{Ind}_{\mathfrak{h}'}^{\mathfrak{h}}\mathbb{C}_{\lambda}.
\]
Then $\dim W_{\lambda}=2^{\left\lfloor \frac{h(\lambda)+1}{2}\right\rfloor }$
where $h(\lambda)$ is the number of nonzero $\lambda_{1},\ldots,\lambda_{n}$.
The following lemma is standard (see for example \cite[Lem. 1.42]{CW}).
\begin{lem}
\label{lem:W_lambda} The $\mathfrak{h}$-module $W_{\lambda}$ is
irreducible. Moreover, every finite-dimensional irreducible $\mathfrak{h}$-module
is isomorphic to $W_{\lambda}$ for some $\lambda\in\mathfrak{h}_{\bar{0}}$. 
\end{lem}

A simple finite-dimensional $\mathfrak{q}(n)$-module $L(\lambda)$
has $W_{\lambda}$ as its highest weight space. Moreover, a highest-weight
module $L(\lambda)$ is finite-dimensional if and only if $\lambda=\sum_{i=1}^{n}\lambda_{i}\varepsilon_{i}$
where $\lambda_{i}-\lambda_{i+1}\in\mathbb{Z}_{+}$ for every $i=1,\ldots,n-1$,
and $\lambda_{i}=\lambda_{i+1}$ implies $\lambda_{i}=0$ (see \cite[Thm. 4]{Pe}
and \cite[Thm 2.18]{CW}). We denote the set of these highest weights
by $\tilde{\Lambda}_{n}$. Any finite-dimensional simple $\mathfrak{q}(n)$-module
is a highest weight module. A $\mathfrak{q}(n)$-module can be integrated
to the Lie supergroup $Q(n)$ if and only if $\lambda\in\tilde{\Lambda}_{n}\cap\mathbb{Z}^{n}$.
Let $\Lambda_{n}:=\tilde{\Lambda}_{n}\cap\mathbb{Z}^{n}$. 

Given a finite-dimensional $\mathfrak{q}(n)$-module $M$ with weight
space decomposition $M=\oplus_{\lambda\in\mathfrak{h}_{\bar{0}}}M_{\lambda}$
where 
\[
M_{\lambda}=\left\{ v\in M\mid xv=\lambda(x)v\text{ for all }x\in\mathfrak{h}_{\bar{0}}\right\} ,
\]
we define the character of $M$ to be
\[
\text{ch}M:=\sum_{\lambda\in\mathfrak{h}_{\bar{0}}}\left(\dim M_{\lambda}\right)e^{\lambda}.
\]
For a Lie superalgebra $\mathfrak{g}$ (resp. Lie supergroup $G$),
we denote by $J(\mathfrak{g})$ (resp. $J(G)$) the ring of all characters
of finite-dimensional $\mathfrak{g}$-modules (resp. $G$-modules).
We shall use the notation $x_{i}:=e^{\varepsilon_{i}}$, $i=1,\ldots,n$. 

An element in the ring of characters is also called a \emph{virtual
character}. Note that the ring of characters is isomorphic to the
quotient of Grothendieck ring by the ideal generated by $[M]-[\Pi M]$
where $\Pi$ is the parity shift functor.

\subsection{The Euler characteristic $E(\lambda)$.}

Let $\lambda\in\Lambda_{n}$. Note that $W_{\lambda}$ is a module
over the Cartan subgroup $H$ that can be extended naturally to the
Borel subgroup $B$ corresponding to $\Phi^{+}$. We have the higher
cohomology modules
\[
H^{i}(\lambda):=H^{i}(G/P,\mathcal{L}(W_{\lambda})).
\]
Here $P$ is the largest parabolic subgroup of $G$ to which the $B$-module
$W_{\lambda}$ can be lifted, and $\mathcal{L}(W_{\lambda})$ denotes
the $G$-equivariant $\mathcal{O}_{G/P}$-module induced by the $P$-module
$W_{\lambda}$. Since $H^{i}(\lambda$) is finite dimensional and
is zero for large enough $i$, the Euler characteristic 
\[
E(\lambda):=\sum_{i\ge0}(-1)^{i}\left[H^{i}(\lambda)\right]
\]
belongs to $J(\mathfrak{q}(n))$.

One has 
\[
E(\lambda)=2^{\left[\frac{h(\lambda)+1}{2}\right]}R^{-1}\cdot\sum_{w\in W}(-1)^{l(w)}w\left(\frac{e^{\lambda}}{\prod_{\beta\in\Phi^{+}(\lambda)}\left(1+e^{-\beta}\right)}\right).
\]

Recall the definition of Schur's $P$-function $p_{\lambda}$ for
$\lambda\in\Lambda_{n}$:

\begin{equation}
p_{\lambda,n}=\sum_{w\in S_{n}/S_{\lambda}}w\left(x^{\lambda}\prod_{\stackrel{1\le i<j\le n}{\lambda_{i}>\lambda_{j}}}\frac{1+x_{i}^{-1}x_{j}}{1-x_{i}^{-1}x_{j}}\right)\label{eq:CK definition of p_lambda}
\end{equation}
where $S_{n}/S_{\lambda}$ denotes the set of minimal length coset
representatives for the stabilizer of $\lambda$ in $S_{n}$. Then
by \cite[Prop. 1]{PS}, $E(\lambda)=2^{\left[\frac{h(\lambda)+1}{2}\right]}\cdot p_{\lambda}$.
By \cite[Cor. 4.26]{B}, $\left\{ E(\lambda)\mid\lambda\in\Lambda_{n}\right\} $
forms a basis to the character ring.

\section{The ring of characters for queer Lie supergroups.}

In this section we describe the ring of characters $Q(n)$ and deduce
Theorem \ref{main theorem} as well as the description for the ring
of character of $SQ(n)$.
\begin{prop}
\label{prop:ring-over-Z}The ring of characters of $Q(n)$ is isomorphic
to the ring 
\begin{equation}
J_{n}=\left\{ f\in\left(\bigoplus_{\lambda\in\mathbb{Z}^{n}}2^{\left\lfloor \frac{h(\lambda)+1}{2}\right\rfloor }\mathbb{Z}x^{\lambda}\right)^{S_{n}}:f\mid_{x_{1}=-x_{2}=t}\text{ is independent in }t\right\} \label{eq: J_n def}
\end{equation}
 where $x^{\lambda}:=x_{1}^{\lambda_{1}}\cdots x_{n}^{\lambda_{n}}$
for $\lambda\in(\lambda_{1},\ldots,\lambda_{n})$ and $h(\lambda)$
is the number of nonzero $\lambda_{i}$'s. 
\end{prop}

The following lemma shows that the ring of characters of $Q(n)$ and
$SQ(n)$ are contained in $J_{n}$ (see (\ref{eq: J_n def})). The
proof is similar to \cite[Prop. 4.3]{SV} except one replaces supercharacters
by characters. 
\begin{lem}
\label{lem:characters are supersymmetric}Let $M$ be a module over
$Q(n)$ or $SQ(n)$ then $\text{ch}M$ is in $J_{n}$.
\end{lem}

\begin{proof}
The character of $M$ is in $\bigoplus_{\lambda\in\mathbb{Z}^{n}}2^{\left\lfloor \frac{h(\lambda)+1}{2}\right\rfloor }\mathbb{Z}x^{\lambda}$
by Lemma \ref{lem:W_lambda}, and the $S_{n}$-symmetry follows from
restricting the representation of $f$ to the even part which is either
$GL(n)$ or $SL(n)$. 

For the evaluation property, restrict $M$ to $\mathfrak{g}_{12}:=\mathfrak{h}_{\bar{0}}\oplus\mathfrak{g}_{\bar{1},\varepsilon_{1}-\varepsilon_{2}}\oplus\mathfrak{g}_{\bar{1},\varepsilon_{2}-\varepsilon_{1}}$.
Then $M$ has composition factors of the form $e^{\lambda}+e^{\lambda-(\varepsilon_{1}-\varepsilon_{2})}$
(typical case) or $e^{\mu}$ where $(\mu,\varepsilon_{1}+\varepsilon_{2})=0$
(atypical case). The former is equal to $x_{1}^{\lambda_{1}}\cdots x_{n}^{\lambda_{n}}\left(1+\frac{x_{2}}{x_{1}}\right)$
and the evaluation $x_{1}=-x_{2}=t$ equals zero. The latter is of
the form $e^{\mu_{1}\left(\varepsilon_{1}+\varepsilon_{2}\right)+\mu_{3}\varepsilon_{3}+\ldots+\mu_{n}\varepsilon_{n}}$
and the evaluation gives equals to $-e^{\mu_{3}\varepsilon_{3}+\ldots+\mu_{n}\varepsilon_{n}}$. 
\end{proof}
Note that since an element $f\in J_{n}$ is a $W$-invariant, the
evaluation $f\mid_{x_{n-1}=-x_{n}=t}$ is also independent of $t$.
Set $J_{0}=\mathbb{Z}$.
\begin{defn}
The evaluation map $\text{ev}:J_{n}\rightarrow J_{n-2}$ is defined
by $\text{ev}(f)=f\mid_{x_{n-1}=-x_{n}}$. 
\end{defn}

\subsection{The kernel of the evaluation map.}

The following proposition is a generalization of the fact that the
kernel of the evaluation map of $GL(m|n)$ is spanned by characters
of Kac modules \cite[Thm. 17]{HR}.
\begin{prop}
\label{kernel} The kernel of the evaluation map is in $J\left(Q(n)\right)$. 
\end{prop}

\begin{proof}
Suppose that $\text{ev}(f)=0$ for $f\in J_{n}$. Then $f$ is divisible
by $x_{n-1}+x_{n}$ and by the $W$-invariance, also divisible by
$\prod_{i<j}\left(x_{i}+x_{j}\right)=\prod_{i<j}\left(e^{\varepsilon_{i}}+e^{\varepsilon_{j}}\right)$.
So
\[
f=\prod_{i<j}\left(e^{\varepsilon_{i}}+e^{\varepsilon_{j}}\right)\cdot g.
\]
Since $f$ and $\prod_{i\ne j}\left(e^{\varepsilon_{i}}+e^{\varepsilon_{j}}\right)$
are $W$-invariant, so is $g$ and by theory of symmetric functions,
$g=\sum^{\text{finite}}a_{\lambda}s_{\lambda}$ where $a_{\lambda}\in\mathbb{Z}$
and $s_{\lambda}$ is the Schur Laurent polynomial
\[
s_{\lambda}=e^{-\rho_{0}}R_{\bar{0}}^{-1}\cdot\sum_{w\in W}(-1)^{l(w)}w\left(e^{\lambda+\rho_{0}}\right)
\]
where $\lambda=(\lambda_{1},\ldots,\lambda_{n})$ and $\lambda_{1}\ge\ldots\ge\lambda_{n}$.
Now
\begin{align*}
f & =\sum_{\lambda}^{\text{finite}}a_{\lambda}\prod_{i\ne j}\left(e^{\varepsilon_{i}}+e^{\varepsilon_{j}}\right)e^{-\rho_{0}}R_{\bar{0}}^{-1}\cdot\sum_{w\in W}(-1)^{l(w)}w\left(e^{\lambda+\rho_{0}}\right)\\
 & =\sum_{\lambda}^{\text{finite}}a_{\lambda}\prod_{i\ne j}\left(1+e^{-\left(\varepsilon_{i}-\varepsilon_{j}\right)}\right)R_{\bar{0}}^{-1}\cdot\sum_{w\in W}(-1)^{l(w)}w\left(e^{\lambda+\rho_{0}}\right)\\
 & =\sum_{\lambda}^{\text{finite}}a_{\lambda}p_{\lambda+\rho_{0}}.
\end{align*}
Note that the last equality uses the fact that $\lambda+\rho_{0}$
has a trivial stabilizer in $W$, that is $\lambda+\rho_{0}\in\Lambda_{n}$.
Now, since $f\in J_{n}$, the maximal $\lambda$ for which $a_{\lambda}\ne0$
is divisible by $2^{\left\lfloor \frac{h(\lambda)+1}{2}\right\rfloor }$.
Subtracting $a_{\lambda}p_{\lambda+\rho_{0}}$ from $f$, we can show
that all $a_{\mu}$'s are divisible by $2^{\left\lfloor \frac{h(\mu)+1}{2}\right\rfloor }$.
Thus $f$ is in the $\mathbb{Z}$-span of the $E(\lambda)$'s and
the assertion follows.
\end{proof}

\subsection{Surjectivity of the evaluation map.}
\begin{prop}
\label{prop:surjectivity integer case}The evaluation map $\text{ev}:J(Q(n))\rightarrow J\left(Q(n-2)\right)$
is surjective.
\end{prop}

To prove the proposition, we give for every $p_{\mathring{\lambda},n-2}$
a Schur function $p_{\lambda,n}$ such that $\text{ev}\left(p_{\lambda,n}\right)=p_{\mathring{\lambda},n-2}$.
For the convenience of the reader, we first present the argument on
an example.
\begin{example}
Suppose $\mathring{\lambda}=(3,1)$. Let $\lambda=\left(3,1,0,0\right)$.
Let us show that $\text{ev}\left(p_{\lambda,4}\right)=p_{\mathring{\lambda},2}$.
{\footnotesize{}}If $w\in S_{4}/S_{\lambda}$ is such that the terms
$1+x_{3}^{-1}x_{4}$ or $1+x_{4}^{-1}x_{3}$ appear in $w\left(x^{\lambda}\prod_{\stackrel{1\le i<j\le n}{\lambda_{i}>\lambda_{j}}}\frac{1+x_{i}^{-1}x_{j}}{1-x_{i}^{-1}x_{j}}\right)$,
then the evaluation of the term is zero. This means that for the evaluation
to be nonzero $w^{-1}\left(\left\{ 3,4\right\} \right)=\left\{ 3,4\right\} $,
which in this example means that $w\in S_{2}:=\left\{ 1,(1\ 2)\right\} $.
Then {\footnotesize{}
\begin{align*}
\text{ev}\left(p_{\lambda,4}\right) & =\text{ev}\sum_{w\in S_{4}/S_{\lambda}}w\left(x^{\lambda}\prod_{\stackrel{1\le i<j\le n}{\lambda_{i}>\lambda_{j}}}\frac{1+x_{i}^{-1}x_{j}}{1-x_{i}^{-1}x_{j}}\right)\\
 & =\text{ev}\sum_{w\in S_{2}}w\left(x^{\lambda}\prod_{\stackrel{1\le i<j\le n}{\lambda_{i}>\lambda_{j}}}\frac{1+x_{i}^{-1}x_{j}}{1-x_{i}^{-1}x_{j}}\right)\\
 & =\sum_{w\in S_{2}}w\left(\text{ev}\left(x^{\lambda}\prod_{\stackrel{1\le i<j\le n}{\lambda_{i}>\lambda_{j}}}\frac{1+x_{i}^{-1}x_{j}}{1-x_{i}^{-1}x_{j}}\right)\right)\\
\\
 & =\sum_{w\in S_{2}}w\left(\text{ev}\left(x_{1}^{3}x_{2}\frac{1+x_{1}^{-1}x_{2}}{1-x_{1}^{-1}x_{2}}\cdot\frac{1+x_{1}^{-1}t}{1-x_{1}^{-1}t}\cdot\frac{1-x_{1}^{-1}t}{1+x_{1}^{-1}t}\cdot\frac{1+x_{2}^{-1}t}{1-x_{2}^{-1}t}\cdot\frac{1-x_{2}^{-1}t}{1+x_{2}^{-1}t}\right)\right)=p_{\mathring{\lambda},2}.\\
\end{align*}
}{\footnotesize\par}
\end{example}

\begin{proof}[Proof of Proposition \ref{prop:surjectivity integer case}.]
 For $n=2$, the evaluation map is surjective since $1\in J(\mathfrak{q}(2))$.
Suppose $n\ge3$. Let $\mathring{\lambda}\in\mathbb{Z}_{+}^{n-2}$
and extend it to $\lambda\in\mathbb{Z}_{+}^{n}$ by adding zeros,
namely $\left\{ \mathring{\lambda}_{1},\ldots,\mathring{\lambda}_{n-2}\right\} \cup\left\{ 0\right\} =\left\{ \lambda_{1},\ldots,\lambda_{n}\right\} $.
We show that $\text{ev}\left(E(\lambda)\right)=E(\mathring{\lambda})$.
Since $E(\lambda)$, $\lambda\in\Lambda_{n-2}$ span $J(Q(n-2))$
over $\mathbb{Z}$, this implies surjectivity. 

Let $k\in\left\{ 1,\ldots,n-1\right\} $ be such that $\lambda_{k}=\lambda_{k+1}=0$
(such $k$ may not be unique). Denote $\sigma_{k}=(k\:n-1)(k+1\:n)\in S_{n}$
and let $\text{ev}_{k}$ be a map on $J_{n}$ such that $\text{ev}_{k}(f)=f\mid_{x_{k}=-x_{k+1}}$.
Then by symmetry $\text{ev}=\sigma_{k}\circ\text{ev}_{k}$. Now, if
$w\in S_{n}/S_{\lambda}$ is such that $w^{-1}\left(\left\{ k,k+1\right\} \right)\subseteq\left\{ \left\{ i,j\right\} \,\mid\,\lambda_{i}\ne\lambda_{j}\right\} $
then 
\[
\text{ev}_{k}\left(w\left(x^{\lambda}\prod_{\stackrel{1\le i<j\le n}{\lambda_{i}>\lambda_{j}}}\frac{1+x_{i}^{-1}x_{j}}{1-x_{i}^{-1}x_{j}}\right)\right)=0.
\]
Denote by $A$ the set of $w\in S_{n}/S_{\lambda}$ for which $w^{-1}\left(\left\{ k,k+1\right\} \right)\subseteq\left\{ i,j\,\mid\,\lambda_{i}=\lambda_{j}\right\} $.
By the definition of $\Lambda^{n}$, $\lambda_{i}=\lambda_{j}$ only
if $\lambda_{i}=0$. Then, up to $S_{\lambda},$ $w^{-1}\left(\left\{ k,k+1\right\} \right)=\left\{ k,k+1\right\} $.
Hence, $w$ preserves $\left\{ k,k+1\right\} $ and consequently $w$
preserves $\left\{ 1,\ldots,n\right\} /\left\{ k,k+1\right\} $. In
particular, $w$ commutes with $\text{ev}_{k}$ and $A$ is the permutation
group of $\left\{ 1,\ldots,n\right\} /\left\{ k,k+1\right\} $. Thus{\footnotesize{}
\begin{align*}
\text{ev}\left(p_{\lambda}\right) & =\sigma_{k}\circ\text{ev}_{k}\left(p_{\lambda,n}\right)\\
 & =\sigma_{k}\circ\text{ev}_{k}\sum_{w\in A}w\left(x^{\lambda}\prod_{\stackrel{1\le i<j\le n}{\lambda_{i}>\lambda_{j}}}\frac{1+x_{i}^{-1}x_{j}}{1-x_{i}^{-1}x_{j}}\right)\\
 & =\sigma_{k}\sum_{w\in A}w\left(\text{ev}_{k}\left(x^{\lambda}\prod_{\stackrel{1\le i<j\le n}{\lambda_{i}>\lambda_{j}}}\frac{1+x_{i}^{-1}x_{j}}{1-x_{i}^{-1}x_{j}}\right)\right)\\
 & =\sigma_{k}\sum_{w\in A}w\left(\text{ev}_{k}\left(x^{\lambda}\prod_{\stackrel{i<j,\ i,j\notin\left\{ k,k+1\right\} }{\lambda_{i}>\lambda_{j}}}\frac{1+x_{i}^{-1}x_{j}}{1-x_{i}^{-1}x_{j}}\cdot\prod_{\stackrel{i<k}{\lambda_{i}\ne0}}\frac{1+x_{i}^{-1}x_{k}}{1-x_{i}^{-1}x_{k}}\cdot\frac{1+x_{i}^{-1}x_{k+1}}{1-x_{i}^{-1}x_{k+1}}\cdot\prod_{\stackrel{i>k}{\lambda_{i}\ne0}}\frac{1+x_{k}^{-1}x_{i}}{1-x_{k}^{-1}x_{i}}\cdot\frac{1+x_{k+1}^{-1}x_{i}}{1-x_{k+1}^{-1}x_{i}}\right)\right)\\
 & =\sigma_{k}\sum_{w\in A}w\left(x^{\mathring{\lambda}}\prod_{\stackrel{i<j,\ i,j\notin\left\{ k,k+1\right\} }{\lambda_{i}>\lambda_{j}}}\frac{1+x_{i}^{-1}x_{j}}{1-x_{i}^{-1}x_{j}}\right)\\
 & =\sum_{w\in S_{n-2}/S_{\mathring{\lambda}}}w\left(x^{\mathring{\lambda}}\prod_{\stackrel{1\le i<j\le n-2}{\lambda_{i}>\lambda_{j}}}\frac{1+x_{i}^{-1}x_{j}}{1-x_{i}^{-1}x_{j}}\right)=p_{\mathring{\lambda},n-2}
\end{align*}
}Since $h(\lambda)=h(\mathring{\lambda})$, we obtain the assertion. 
\end{proof}

\subsection{The ring of characters.}

We can now prove the main theorem.
\begin{proof}[Proof of Theorem \ref{main theorem}]
 We prove the theorem by induction on $n$. For $n=0$, $J_{n}=\mathbb{Z}$
by definition. For $n=1$, $Q(1)$ admits a $1$-dimensional representation
whose character is a scalar multiple of $e^{m\varepsilon_{1}}$ for
every $m\in\mathbb{Z}$. Then $J\left(Q(1)\right)=\mathbb{Z}\left[x_{1}^{\pm1}\right]=J_{1}$. 

Suppose that $J(Q(n-2))=J_{n-2}$. Recall the evaluation map $\text{ev}:J_{n}\rightarrow J_{n-2}$
and that $J(Q(n))\subseteq J_{n}$. By Proposition \ref{prop:surjectivity integer case},
$ev\left(J(Q(n))\right)=J(Q(n-2))=J_{n-2}$. So, any element in $J_{n}$
is a sum of an element from $J(Q(n))$ and an element in the kernel
of $\text{ev}$. By Proposition \ref{kernel}, the kernel is also
contained in $J(Q(n))$ and the assertion is obtained. 
\end{proof}
\begin{cor}
\label{cor:SQ(n)}The ring of characters of $SQ(n)$ is also isomorphic
to $J_{n}$. Indeed, since $SQ(n)$ is a subgroup of $Q(n)$, every
$Q(n)$-module restricts to $SQ(n)$ and so $J(SQ(n))\supseteq J(Q(n))=J_{n}$.
On the other hand, by Lemma \ref{lem:characters are supersymmetric},
$J(SQ(n))\subseteq J_{n}$ and we obtained the equality. In fact,
the restriction of simple $Q(n)$-modules to $SQ(n)$ is known by
\cite[Cor. 4.6]{GS}.
\end{cor}

\begin{cor}
\label{cor:PQ and PSQ}The rings of characters of $PQ(n)$ and $PSQ(n)$
are equal to 
\[
J_{n}\cap\text{span}_{\mathbb{Z}}\left\{ x_{1}^{\lambda_{1}}\cdots x_{n}^{\lambda_{n}}\mid\sum_{i=1}^{n}\lambda_{i}=0\right\} .
\]
Indeed, a module over $Q(n)$ (resp. $SQ(n)$) is lifts to a module
over $PQ(n)$ (resp. $PSQ(n)$) if and only if the identity matrix
acts by zero. This holds if and only if all the weights $\lambda$
of the module satisfy that $\sum_{i=1}^{n}\lambda_{i}=0$. When the
module is simple, this is true if and only if one weight $\lambda$
of the module satisfies that $\sum_{i=1}^{n}\lambda_{i}=0$. 
\end{cor}

\begin{rem}
The category of finite-dimensional $Q(n)$-modules admits a subcategory
of polynomial representations, namely those which appear as compositions
factors in the tensor algebra of the natural representations. The
Grothendieck ring of this subcategory is the set of all polynomials
in $J_{n}$. This ring is isomorphic to the center of $U(\mathfrak{q}(n))$
and was described in \cite[Thm. 2.11]{Pa} (see also \cite[A.3.4]{CW}).
\end{rem}

\section{The ring of characters for queer Lie superalgebras}

We first describe the ring of characters for the category of finite-dimensional
modules whose weights are half integers. The simple modules in this
category were studied in \cite{CK}.
\begin{prop}
Let $J_{n}$ be as in (\ref{eq: J_n def}) and 
\[
J_{\frac{1}{2},n}=\left(\text{span}_{\mathbb{Z}}2^{\left\lfloor \frac{n+1}{2}\right\rfloor }x_{1}^{k_{1}}\cdots x_{n}^{k_{n}}\ \mid\ k_{1},\ldots,k_{n}\in\frac{1}{2}+\mathbb{Z}\right)^{S_{n}}.
\]
The ring of characters of the category of finite-dimensional half-integer
weights is isomorphic to
\[
J_{n}\oplus\prod_{1\le i<j\le n}\left(x_{i}+x_{j}\right)J_{\frac{1}{2},n}.
\]
\end{prop}

\begin{proof}
The elements in $\prod_{1\le i<j\le n}\left(x_{i}+x_{j}\right)J_{\frac{1}{2},n}$
are in the kernel of the evaluation map and so belong to the ring
of characters for the same reasoning as in Proposition \ref{kernel}. 

For the other inclusion, we need to show that the character of of
any module $M$ a half-integer (non-integer) weights is divisible
by $\left(x_{i}+x_{j}\right)$ for every $i\ne j$. Note that $\alpha=\varepsilon_{i}-\varepsilon_{j}$
is also an even root. Restrict $M$ to the corresponding $\mathfrak{sl}(2)$-triple
$\mathfrak{s_{\alpha}}$. Since the weights are not integers, zero
is not a weight and the restriction of $\mathfrak{s}_{\alpha}$ is
a direct sum of $\mathfrak{sl}(2)$-strings of even length. In particular,
as an $\mathfrak{h}_{\bar{0}}$-module, $M=\bigoplus_{\lambda\in B}N_{\lambda}$
where $B$ is a finite subset of $\mathfrak{h}_{\bar{0}}$ and $\text{ch}N_{\lambda}=e^{\lambda}+e^{\lambda-\alpha}=x^{\lambda}+x^{\lambda}x_{i}^{-1}x_{j}=(x_{i}+x_{j})x^{\lambda}x_{i}^{-1}$. 
\end{proof}
Similarly to $J_{\frac{1}{2},n},$ we can define the vector space
\[
J_{a,n}:=\left(\text{span}_{\mathbb{Z}}\left\{ 2^{\left\lfloor \frac{n+1}{2}\right\rfloor }x_{1}^{k_{1}}\cdots x_{n}^{k_{n}}\ \mid\ k_{1},\ldots,k_{n}\in a+\mathbb{Z}\right\} \right)^{S_{n}}
\]
 for any $a\in\mathbb{C}/\mathbb{Z}$. 
\begin{prop}
The ring of characters of finite-dimensional $\mathfrak{q}(n)$-modules
is isomorphic to
\[
J_{n}\oplus\bigoplus_{a\in\mathbb{C}/\mathbb{Z}}\prod_{1\le i<j\le n}\left(x_{i}+x_{j}\right)J_{a,n}.
\]
\end{prop}

\begin{proof}
Similarly to the previous proposition the elements in $\prod_{1\le i<j\le n}\left(x_{i}+x_{j}\right)J_{a,n}$
belong to the ring of characters for the same reasoning as they are
in the kernel of the evaluation map. 

For the other inclusion, let $L(\lambda)$ be a simple highest weight
module for which $\lambda\in a+\mathbb{Z}^{n}$. Suppose $a\ne\frac{1}{2}$,
then, $\lambda_{i}\ne-\lambda_{j}$ for every $i\ne j$ and $L(\lambda)$
is a typical module. By, \cite[Thm. 2]{Pe} (see also \cite[Thm. 2.5.2.]{CW}),
\[
\text{ch}L(\lambda)=2^{\left\lfloor \frac{n+1}{2}\right\rfloor }R\cdot\sum_{w\in W}(-1)^{l(w)}e^{w(\lambda)}.
\]
In particular $\text{ch}L(\lambda)$ is divisible by $e^{\rho_{0}}\prod_{i<j}\left(1+e^{-\varepsilon_{i}+\varepsilon_{j}}\right)=\prod_{i<j}\left(x_{i}+x_{j}\right)$.
For $a=\frac{1}{2}$, the characters are divisible by $\prod_{i<j}\left(x_{i}+x_{j}\right)$
by the previous proposition and the assertion follows.
\end{proof}
\begin{rem}
Note that $\mathfrak{q}(n)$ and $\mathfrak{sq}(n)$ admit the same
rings of characters by a similar argument as in Corollary \ref{cor:SQ(n)}.
\end{rem}

\begin{cor}
Similarly to Corollary \ref{cor:PQ and PSQ}, 
\begin{align*}
J\left(\mathfrak{pq}(n)\right) & =J\left(\mathfrak{q}(n)\right)\cap\text{span}_{\mathbb{Z}}\left\{ x^{\lambda}\mid\lambda\in\mathbb{Z}^{n},\sum_{i=1}^{n}\lambda_{i}=0\right\} \\
J\left(\mathfrak{spq}(n)\right) & =J\left(\mathfrak{sq}(n)\right)\cap\text{span}_{\mathbb{Z}}\left\{ x^{\lambda}\mid\lambda\in\mathbb{Z}^{n},\sum_{i=1}^{n}\lambda_{i}=0\right\} .
\end{align*}
\end{cor}

\section{The Weyl Groupoid for $\mathfrak{q}(n)$. }

We describe the polynomial invariants of certain affine action of
the super Weyl groupoid $\mathfrak{W}$ of $\mathfrak{q}(n)$. We
follow \cite[Sec. 9]{SV} (for the periplectic case, see \cite[Sec. 5.4]{IRS}).

We identify $\mathfrak{h}_{0}^{*}$ with its dual using the scalar
product in which $(\varepsilon_{i},\varepsilon_{j})=\delta_{ij}$.
Let $\mathfrak{T}$ be a groupoid with the set of objects $\left\{ \left[\varepsilon_{i}+\varepsilon_{j}\right],\left[-\varepsilon_{i}-\varepsilon_{j}\right]:i<j\right\} $
and the set of morphisms $\text{Mor}_{\mathfrak{T}}\left([\alpha],[\beta]\right)=\emptyset$
if $\alpha\ne\pm\beta$, and $\text{Mor}_{\mathfrak{T}}\left([\alpha],[\pm\alpha]\right)$
is of cardinality $1$. For every object $[\alpha]$, we denote by
$r_{\alpha}$ the unique morphism in $\text{Mor}_{\mathfrak{T}}\left([\alpha],[-\alpha]\right)$,
and impose the relation $r_{\alpha}r_{-\alpha}=id_{[-\alpha]}.$

The Weyl group $W=S_{n}$ acts on $\Delta(\mathfrak{g}_{0})$ and
we have the naturally defined homomorphism $\Phi$ from $W$ to the
group of autoequivalences of $\mathfrak{T}$. This yields the semidirect
product groupoid $\widetilde{\mathfrak{T}}=W\ltimes\mathfrak{T}$.
The objects of $\widetilde{\mathfrak{T}}$ are the same as the objects
of $\mathfrak{T}$. The morphisms in $\widetilde{\mathfrak{T}}$ are
generated by $r_{\alpha}$ and $w_{\alpha}^{w\alpha}\in\text{Mor}_{\widetilde{\mathfrak{T}}}\left([\alpha],[w\alpha]\right)$,
where $w\in W$ and the following diagrams commute:\[   \begin{tikzcd}     
\left[\alpha\right] \arrow{rr}{\left(uw\right)_{\alpha}^{uw\alpha}} \arrow[swap]{dr}{w_{\alpha}^{w\alpha}}    &  & \left[uw\alpha\right]  \\      
&  \left[w\alpha\right] \arrow[swap]{ur}{u_{w\alpha}^{uw\alpha}} &   \end{tikzcd}\qquad
\begin{tikzcd}     
\left[\alpha\right] \arrow{r}{\left(w\right)_{\alpha}^{w\alpha}} \arrow[swap]{d}{r_{\alpha}}     & \left[w\alpha\right]\arrow{d}{r_{-\alpha}}  \\      
  \left[-\alpha\right] \arrow[swap]{r}{w_{-\alpha}^{-w\alpha}} &  \left[-w\alpha\right] \end{tikzcd}\]

The Weyl groupoid is defined as 
\[
\mathfrak{W}:=W\amalg\widetilde{\mathfrak{T}},
\]
where $W$ is considered as a groupoid with a single point base $[W]$.
\begin{example}
For $\mathfrak{q}(2)$, let $s$ be the nontrivial element of $S_{2}$.
Then $\mathfrak{W}$ takes the following form
\end{example}

\[   \begin{tikzcd}     
\left[S_2\right] 
\arrow[loop left, distance=2.5em, start anchor={[yshift=-1ex]west}, end anchor={[yshift=1ex]west}]{}{s}
 \end{tikzcd}\qquad
\begin{tikzcd}     
\left[\varepsilon_1-\varepsilon_2\right] 
\arrow[swap,bend right=30]{d}{s} 
\arrow[swap,bend right=75]{d}{r_{\varepsilon_2-\varepsilon_1,\varepsilon_1-\varepsilon_2}}       \\      
  \left[\varepsilon_2-\varepsilon_1\right] 
\arrow[swap,bend right=75]{u}{r_{\varepsilon_2-\varepsilon_2,\varepsilon_1-\varepsilon_2}}
\arrow[swap,bend right=30]{u}{s} 
\end{tikzcd}
\]

Let $\mathfrak{H}$ denote the category of all affine subspaces of
$\mathfrak{h}$ with morphisms given by affine transformations. For
any $\alpha=\pm(\varepsilon_{i}-\varepsilon_{j})$, let 
\[
\Pi_{\alpha}=\left\{ (x_{1},\ldots,x_{n})\in\mathbb{C}^{n}\ :\ x_{i}=-x_{j}\right\} .
\]
Note that $\Pi_{\alpha}=\Pi_{-\alpha}$. Define  $\tau_{\alpha}\in\text{Mor}_{\mathfrak{H}}(\Pi_{\alpha},\Pi_{\alpha})$
 by the formula
\[
\tau_{\alpha}(x_{1},\ldots,x_{n})=(x_{1},\ldots,x_{i}+1,\ldots,x_{j}-1,\ldots,x_{n}).
\]
  Define the functor  $F:\mathfrak{W}\to\mathfrak{H}$  by setting
\[
F([\alpha])=\Pi_{\alpha},\quad F([W])=\mathbb{C}^{n},\quad F(r_{\alpha})=\tau_{\alpha},\quad F(w)=w,\quad F(w_{\alpha,w\alpha})=\operatorname{Res}_{\Pi_{\alpha}}w.
\]
A function  $f$  on  $\mathfrak{h}$  is called  $\mathfrak{W}$-invariant
if for any  $\varphi\in\text{Mor}_{\mathfrak{W}}(A,B)$    
\[
F(\varphi)^{*}\operatorname{Res}_{FB}f=\operatorname{Res}_{FA}f.
\]

The description of  $J(Q(n))$  can be formulated as follows:
\begin{thm}
The $\mathbb{Q}$-span of the ring of characters $J(Q(n))$  of finite-dimensional
representations of the supergroup  $Q(n)$  is isomorphic to the ring
 $\left(\mathbb{Q}[x_{1}^{\pm1},\ldots,x_{n}^{\pm1}]\right)^{\mathfrak{W}}$
 of invariants under the Weyl groupoid  $\mathfrak{W}$ as defined
above. 
\end{thm}

\end{document}